\documentclass[a4paper]{amsart}

\usepackage{amsmath,amssymb,amsthm}
\usepackage[all]{xy}
\usepackage[dvipdfmx]{graphicx}
\usepackage{tikz}
\usepackage[dvipdfmx,colorlinks=true,backref=page]{hyperref}
\usepackage{spectralsequences}

\newtheorem{theorem}{Theorem}[section]
\newtheorem{proposition}[theorem]{Proposition}
\newtheorem{lemma}[theorem]{Lemma}
\newtheorem{corollary}[theorem]{Corollary}

\theoremstyle{definition}
\newtheorem{definition}[theorem]{Definition}
\newtheorem{example}[theorem]{Example}
\newtheorem{conjecture}[theorem]{Conjecture}

\newtheorem{question}[theorem]{Question}

\theoremstyle{remark}
\newtheorem{remark}[theorem]{Remark}

\numberwithin{equation}{section}

\newcommand{\Q}{\mathbb{Q}}
\newcommand{\R}{\mathbb{R}}
\newcommand{\C}{\mathbb{C}}
\newcommand{\cat}{\mathsf{cat}}
\newcommand{\secat}{\mathsf{secat}}

\newcommand{\TC}{\mathsf{TC}}
\newcommand{\MTC}{\mathsf{MTC}}
\newcommand{\HTC}{\mathsf{HTC}}
\newcommand{\nil}{\mathsf{nil}}
\newcommand{\cupl}{\mathsf{cup}}
\newcommand{\zcl}{\mathsf{zcl}}

\title{Rational sequential parametrized topological complexity}

\author[Y. Minowa]{Yuki Minowa}
\address{Department of Mathematics, Kyoto University, Kyoto, 606-8502, Japan}
\email{minowa.yuki.48z@st.kyoto-u.ac.jp}

\date{March 3, 2025}



\begin{document}

\begin{abstract}
  Sequential parametrized topological complexity is a numerical homotopy invariant of a fibration, which arose in the robot motion planning problem with external constraints. In this paper, we study sequential parametrized topological complexity in view of rational homotopy theory. We generalize results on topological complexity, and in particular, give an explicit algebraic upper bound for sequential parametrized topological complexity when a fibration admits a certain decomposition, which is a generalization of the result of Hamoun, Rami and Vandembroucq on topological complexity.
\end{abstract}

\maketitle

%%%%% Section 1 %%%%%

\section{Introduction}\label{Introduction}

\emph{Topological complexity} is a numerical homotopy invariant introduced by Farber \cite{F} to describe the instability in the motion planning problem. Since its introduction, topological complexity has been studied intensely in several contexts, and has been generalized mainly in several directions. Here, we recall three of those generalizations. Rudyak \cite{R} considered the robot motion planning passing through several intermediate points, where topological complexity only considers initial and terminal points. Then he introduced \emph{sequential (or higher) topological complexity}. On the other hand, Cohen, Farber and Weinberger \cite{CFW} considered the robot motion planning having external constraints expressed by a fibration, and introduced \emph{parametrized topological complexity} of a fibration with path-connected fiber. Recently, Farber and Paul \cite{FP} combined the above two generalizations of topological complexity, and introduced \emph{sequential parametrized topological complexity}.

We recall the definition of sequential parametrized topological complexity. Let $p\colon X\to B$ be a fibration with a path-connected fiber.
Let $X^r_B$ denote the fiberwise product of $r$ copies of $X$ over $B$, and let
\[
\mathcal{P}_B X = \{\gamma \colon [0, 1] \to X \mid p \circ \gamma \text{ is a constant map} \}.
\]
be the fiberwise path space. Then the evauation map
\[
\Pi_r\colon \mathcal{P}_BX \to X^r_B,\quad\gamma \mapsto \left(\gamma(0),\gamma(\tfrac{1}{r-1}),\cdots,\gamma(1) \right)
\]
is a fibration. For $r\ge 2$, \emph{the $r$-th sequential parametrized topological complexity} of a fibration $p\colon X\to B$, denoted by $\TC_r[X\to B]$, is defined to be the minimal integer $k$ such that $X^r_B$ is covered by $k+1$ open sets having local section of $\Pi_r$. If such an integer does not exist, then we set $\TC_r[X\to B]=\infty$. Observe that an element of $\mathcal{P}_B X$ may be thought as a motion of a robot in a space $X$ with constraints given by a fibration $p\colon X\to B$, so a section of $\Pi_r$ may be a motion planning with prescribed points in motions under external constraints given by a fibration $p\colon X\to B$. Thus sequential parametrized topological complexity describes the instability in the robot motion planning problem with external constraints, as mentioned above.

In this paper, we study sequential parametrized topological complexity from rational homotopy theory viewpoint. 
As well as topological complexity, there are inequalities
\begin{equation}
  \label{PTC and zcl}
  \zcl_r[X\to B]\le \TC_r[X\to B]\le \cat_B(X^r_B)
\end{equation}
\noindent
where $\zcl_r[X\to B]$ is the cohomological invariant called the $r$-th zero-divisors cup-length and $\cat_B(X^r_B)$ denotes the fiberwise LS-category over $B$. If $r=2$ and $B$ is a point, then these inequalities are well-known ones for topological complexity. In this case, Jessup, Murillo and Parent \cite{JMP} proved that the left inequality is actually an equality whenever $X$ is a formal rational space. Our first purpose is to provide a sufficient condition for the rational sequential parametrized topological complexity to be determined by the rational cohomology.

\begin{theorem}
  \label{PTC nil ker}
  Let $F\to X\to B$ be a fibration of rational nilpotent spaces of finite rational type which is totally noncohomologous to zero. If $F$ is elliptic and formal and $B$ is formal, then
  \[
  \TC_r[X\to B] = \zcl_r[X\to B].
  \]
\end{theorem}

LS category is generally quite hard to compute, and so is fiberwise LS category. Then the upper bound \eqref{PTC and zcl} is not practical. Even worse, it is not sharp, in general. However, in the rational category, Hamoun, Rami and Vandembroucq \cite{HRV} gave an easier and sharper upper bound for topological complexity under some conditions. Our second aim is to generalize this result to sequential parametrized topological complexity. Let $F \to X\to B$ be a fibration of rational spaces. We say that this fibration admits \emph{an odd-degree extension} if there is a homotopy commutative diagram
\[
\xymatrix{
F \ar[r]\ar[d]& X \ar[r]\ar[d]& B\ar@{=}[d] \\
\widehat{F} \ar[r]& \widehat{X} \ar[r]& B
}
\]
such that the bottom row is a fibration of rational spaces and the kernel of the map $\pi_\ast(F)\to\pi_\ast(\widehat{F})$ is finite dimensional and concentrated in odd degrees.

\begin{theorem}
  \label{main 1}
  Let $F \to X\to B$ be a fibration of rational nilpotent spaces of finite rational type which is an odd-degree extension of $\widehat{F}\to\widehat{X}\to B$. Then 
  \[
  \TC_r[X\to B] \le \TC_r[\widehat{X}\to B] + m(r-1).
  \]
  where $m$ is the dimension of the kernel of $\pi_\ast(F)\to\pi_\ast(\widehat{F})$.
\end{theorem}
\noindent
In particular, we determine $\TC_r[X\to B]$ for a rational fibration $F \to X\to B$ such that $F$ is elliptic and concentrated in odd degrees, i.e., $\pi_\mathrm{even}(F)=0$.

\begin{theorem}
  \label{main2}
  Let $F \to X\to B$ be a fibration of rational nilpotent spaces of finite rational type. Suppose that $F$ is elliptic and concentrated in odd degrees. Then 
  \[
  \TC_r[X\to B] = \TC_r(F) = (r-1)\dim(\pi_\mathrm{odd}(F)).
  \]
\end{theorem}

Farber and Opera \cite[Section 8]{FO} introduced a formal power series which describes the growth of the sequential topological complexity of a space. They asked for which spaces the power series converges to a certain type of rational functions, which is still open. Our third purpose is to consider an analogue of this power series. For a fibration $p\colon X\to B$, we define
\[
\mathcal{F}_{p}(z) := \sum_{r=1}^{\infty}\TC_{r+1}[p\colon X\to B]z^{r}.
\]
This power series illustrates the growth of $\TC_{r}[p\colon X\to B]$. The question of Farber and Opera is generalized as follows.

\begin{question}
  \label{PTC growth}
  For what fibration $F\to X\xrightarrow{p} B$ of finite rational type, is $\mathcal{F}_{p}(z)$ a rational function
  \[
  \frac{P(z)}{(1-z)^2}
  \]
  such that $P(z)$ is a polynomial with $P(1)= \cat(F)$?
\end{question}
\noindent
Combining the above results, we will provide two classes of fibrations which support Question \ref{PTC growth}.

This paper is structured as follows. Section \ref{Secat} recalls the definition of sectional category. Section \ref{Rational} provides the algebraic description of rational sequential parametrized topological complexity. Section \ref{formal} investigates the formality of fibrations and proves Theorem \ref{PTC nil ker}. Section \ref{Bound} computes the algebraic upper bound of $\TC_r[X\to B]$ in relation to odd-degree extension, including the proof of Theorem \ref{main 1}. Section \ref{Odd fiber} proves Theorem \ref{main2}. Finally, Section \ref{PTC function} provides some affirmative cases for Question \ref{PTC growth}.

\subsection*{Acknowledgement}
The author is grateful to Daisuke Kishimoto for his valuable advice. The author was supported by JST SPRING, Grant Number JPMJSP2110.

%%%%% Section 2 %%%%%

\section{Sectional category}\label{Secat}

Recall that the sectional category of a map $f\colon X\to Y$, denoted by  $\mathsf{secat}(f)$, is the minimal integer $k$ such that $Y$ is covered by $k+1$ open sets over each of which $f$ has a right homotopy inverse. If such an integer does not exist, then we set $\mathsf{secat}(f)=\infty$. We say that two maps $f_1\colon X_1\to Y_1$ and $f_2\colon X_2\to Y_2$ are equivalent if there is a homotopy commutative diagram
\[
  \xymatrix{
    X_1\ar[r]^{f_1}\ar[d]_\simeq&Y_1\ar[d]^\simeq\\
    X_2\ar[r]^{f_2}&Y_2
  }
\]
where the columns are homotopy equivalences. In this case, $\mathsf{secat}(f_1)=\mathsf{secat}(f_2)$.

By the definition of the sequential parametrized topological complexity, there is an identity
\[
\TC_r[X\to B]=\mathsf{secat}(\Pi_r\colon X^I_B \to X^r_B).
\]
Then since $\Pi_r$ is equivalent to the fiberwise diagonal map $\Delta_r\colon X\to X^r_B$ \cite[Proposition 7.3]{CFW}, we get
\begin{equation}
  \label{PTC Delta}
  \TC_r[X\to B]=\mathsf{secat}(\Delta_r\colon X\to X^r_B).
\end{equation}

\begin{lemma}
  \label{bounds 1}
  There is an inequarity
  \[
  \TC_{r}(F) \le \mathsf{secat}(\Pi_r).
  \]
\end{lemma}

\begin{proof}
  The inequality follows immediately from the pullback diagram
  \[
  \xymatrix{
  F^I \ar[r] \ar[d]_{\Pi_r} & X^I_B \ar[d]^{\Pi_r} \\
  F^{r} \ar[r] & X^r_B.
  }
  \]
\end{proof}

%%%%% Section 3 %%%%%

\section{Rational model}\label{Rational}

In this section, we construct a rational model for sequential parametrized topological complexity. In what follows, we only consider rational spaces  which are path-connected, nilpotent and of finite rational type, unless otherwise specified.

\begin{remark}
  While some authors in the references worked only with simply connected spaces, one can verify that their arguments may be applied to nilpotent spaces of finite rational type. For more about the rational homotopy theory of nilpotent spaces, see \cite{BG, BK, FHT2} for instance.
\end{remark}

We first recall the result of Carrasquel-Vera \cite{CV17} on the algebraic description of the rational sectional category. Let $f\colon X\to Y$ be a map of rational spaces which admits a homotopy retraction $r\colon Y\to X$. Then $f$ has a model $\varphi\colon A\otimes \Lambda W \to A$ which is a retraction of a relative Sullivan algebra $A \to A\otimes \Lambda W$ modelling the map $r$.

\begin{theorem}
  \label{secat FH}
  \cite[Theorem 1]{CV17}
  The sectional category of $f$ is equal to the minimal integer $k$ such that the cdga projection
  \[
  \rho_k\colon A\otimes \Lambda W \twoheadrightarrow \frac{A\otimes \Lambda W}{(\mathrm{Ker}\, \varphi)^{k+1}}
  \]
  admits a homotopy retraction.
\end{theorem}

We then apply Theorem \ref{secat FH} to the sequential parametrized topological complexity by constructing a surjective model for the map $\Pi_r\colon \mathcal{P}_BX \to X^r_B$. Let $F\to X\xrightarrow{p} B$ be a rational fibration with a relative Sullivan model
\[
C \hookrightarrow C\otimes \Lambda V \to \Lambda V.
\]
By \eqref{PTC Delta}, it is sufficient to give a surjective model of the fiberwise diagonal $\Delta_r\colon X\to X^r_B$, of which the first projection $p_1\colon X^r_B\to X$ is a retraction. Then, by arguing as in \cite{CV17}, we obtain a model of $\Delta_r$ by
\begin{equation}
  \label{diag model}
  id\otimes\mu_r\colon C\otimes (\Lambda V)^{\otimes r}\to C\otimes \Lambda V
\end{equation}
where $\mu_r$ denotes the multiplication.

\begin{proposition}
  \label{PTC quot}
  The $r$-th sequential parametrized topological complexity $\TC_r[X\to B]$ is the minimal integer $k$ such that the cdga projection
  \[
  \rho_k\colon C\otimes (\Lambda V)^{\otimes r} \twoheadrightarrow \frac{C\otimes (\Lambda V)^{\otimes r}}{(\mathrm{Ker}\, (id\otimes\mu_r))^{k+1}}
  \]
  admits a homotopy retraction.
\end{proposition}

\begin{proof}
  Let $i_1\colon \Lambda V\to (\Lambda V)^{\otimes r}$ denote the first inclusion. The above model $id\otimes\mu_r$ of $\Delta_r$ is obviously a retraction of $id\otimes i_1\colon C\otimes \Lambda V \to C\otimes (\Lambda V)^{\otimes r}$, which is a model of the first projection $p_1$. Thus we can apply Theorem \ref{secat FH}, finishing the proof.
\end{proof}
\noindent
Conversely, one can define the sequential parametrized topological complexity of relative Sullivan algebras.

\begin{definition}
  For an arbitrary relative Sullivan algebra $C\hookrightarrow C\otimes \Lambda V$, we define $\TC_r[C\hookrightarrow C\otimes \Lambda V]$ by the minimal integer $k$ such that the cdga projection
  \[
  \rho_k\colon C\otimes (\Lambda {V})^{\otimes r} \twoheadrightarrow \frac{C\otimes (\Lambda {V})^{\otimes r}}{(\mathrm{Ker}\, (id\otimes\mu_r))^{k+1}}
  \]
  admits a homotopy retraction.
\end{definition}

We next consider some lower bounds of $\TC_r[X\to B]$, which are special cases of the lower bounds of the rational sectional category originally introduced by Fern\'{a}ndez Su\'{a}rez, Ghienne, Kahl and Vandembroucq \cite{FGKV}.

\begin{definition}
  \label{MTC and HTC}
  Let $\rho_k$ as in Proposition \ref{PTC quot}.
  \begin{enumerate}
    \item The module $r$-th sequential parametrized topological complexity of $X\to B$, denoted by $\MTC_r[X\to B]$, is the minimal integer $k$ such that admits a differential $C\otimes (\Lambda V)^{\otimes r}$-module homotopy retraction.
    \item The homology $r$-th sequential parametrized topological complexity of $X\to B$, denoted by $\HTC_r[X\to B]$, is the minimal integer $k$ such that $H(\rho_k)$ is injective.
    \item The $r$-th zero-divisor cup-length, denoted by $\zcl_r[X\to B]$, is the minimal integer $k$ such that
    \[
      \left( \mathrm{Ker}\, (\Delta_r^*\colon H^*(X^r_B;\Q)\to H^*(X;\Q)) \right) ^{k+1} = 0.
    \]
  \end{enumerate}
\end{definition}
\noindent
Then one obtains the following inequalities.
\begin{equation}
  \label{PTC ineqs}
  \zcl_r[X\to B]\le \HTC_r[X\to B]\le \MTC_r[X\to B] \le \TC_r[X\to B].
\end{equation}

\section{On the formality of fibrations}\label{formal}

In this section, we give a proof of Theorem \ref{PTC nil ker}. In what follows, we only consider the cohomology of rational coefficients unless otherwise specified. 

We first recall the formality of maps. A map $f\colon X\to Y$ between rational spaces is said to be \emph{formal} if there exists a homomorphism $\Lambda W\to \Lambda V$ and a homotopy commutative diagram
\[
\xymatrix{
A_{PL}(Y) \ar[d]_-{A_{PL}(f)}& \Lambda W \ar[l]_-{\simeq}\ar[r]^{\simeq}\ar[d]& H^*(Y) \ar[d]^{H^*(f)}\\
A_{PL}(X) & \Lambda V \ar[l]_-{\simeq}\ar[r]^{\simeq}& H^*(X).
}
\]
\noindent
For a ring $R$, let $\nil \, I$ denote the nilpotency of an augmented ideal $I$ of $R$; namely, $\nil \, I$ is the minimal integer $k$ such that $I^{k+1}=0$ in $R$. Then we have the following lemma.

\begin{lemma}
  \label{secat formal}
  \cite[Theorem 24]{CV15}
  If $f$ is formal, then $\mathsf{nil}\, \mathrm{Ker}\, H^*(f) = \secat(f)$.
\end{lemma}

We now consider a rational fibration $F\xrightarrow{j}X\xrightarrow{p}B$. We aim to show the formality of the fiberwise diagonal $\Delta_r\colon X\to X^r_B$ in order to prove Theorem \ref{PTC nil ker}. We denote the relative minimal Sullivan model of $F\to X\to B$ by $\mathcal{M}\hookrightarrow \mathcal{M}\otimes \Lambda V\twoheadrightarrow \Lambda V$; namely, there is a commutative diagram
\[
\xymatrix{
  \mathcal{M} \ar[r]\ar[d]^{\simeq}_{m_B}& \mathcal{M}\otimes \Lambda V \ar[r]\ar[d]_{\simeq}^{m_X}& \Lambda V\ar[d]^{\overline{m}}_{\simeq} \\
A_{PL}(B) \ar[r]^{A_{PL}(p)}& A_{PL}(X) \ar[r]^{A_{PL}(j)}& A_{PL}(F),
}
\]
where all columns are quasi-isomorphisms.

\begin{lemma}
  \label{diag comm}
  There exists a commutative diagram
  \[
  \xymatrix{
    \mathcal{M}\otimes (\Lambda V)^{\otimes r}\ar[rr]^{id\otimes \mu_r}\ar[d]_{\xi}^{\simeq} && \mathcal{M}\otimes \Lambda V\ar[d]^{m_X}_{\simeq}\\
  A_{PL}(X^r_B) \ar[rr]^{A_{PL}(\Delta_r)}&& A_{PL}(X).
  }
  \]
\end{lemma}

\begin{proof}
  We first show that there is a quasi-isomorphism $\xi\colon \mathcal{M}\otimes (\Lambda V)^{\otimes r}\to A_{PL}(X^r_B)$ such that, for $1 \le \lambda \le r$, there is a commutative diagram
  \[
  \xymatrix{
    \mathcal{M}\otimes \Lambda V \ar[rr]^{id\otimes i_\lambda}\ar[d]^{\simeq}_{m_X}&& \mathcal{M}\otimes (\Lambda V)^{\otimes r}\ar[d]^{\xi}_{\simeq}\\
  A_{PL}(X) \ar[rr]^{A_{PL}(p_\lambda)}&& A_{PL}(X^r_B),
  }
  \]
  where the columns are Sullivan models and $i_\lambda$ and $p_\lambda$ denote the $\lambda$-th injection and projection, respectively. Recall that, for $1 \le \lambda \le r-1$, there is a pullback diagram
  \[
  \xymatrix{
  X^{\lambda+1}_B \ar[r] \ar[d]_{p_{\lambda+1}}& X^{\lambda}_B \ar[d]\\
  X \ar[r]^{p}& B.
  }
  \]
  Then for $\lambda=1$, the pushout construction as in \cite[Proposition 15.8]{FHT} asserts the existence of a homomorphism $\xi_{2}\colon V_B\otimes(\Lambda V)^{\otimes 2}\to A_{PL}(X^2_B)$ and a commutative diagram
  \[
  \xymatrix{
    \mathcal{M}\otimes \Lambda V \ar[rr]^{id\otimes i_1}\ar[d]^{\simeq}_{m_X}&& \mathcal{M}\otimes (\Lambda V)^{\otimes 2}\ar[d]^{\xi_2}_{\simeq}&& \mathcal{M}\otimes \Lambda V \ar[ll]_{id\otimes i_2}\ar[d]_{\simeq}^{m_X}\\
  A_{PL}(X) \ar[rr]^{A_{PL}(p_1)}&& A_{PL}(X^2_B)&& A_{PL}(X)\ar[ll]_{A_{PL}(p_2)}.
  }
  \]
  By continuing this process inductively, we obtain the homomorphism $\xi=\xi_r$.

  Since the composite $p_\lambda\circ\Delta_r$ is the identity, we have
  \begin{align*}
    m_X\circ(id\otimes \mu_r)\circ(id\otimes i_\lambda)
    &= m_X \\
    &= A_{PL}(\Delta_r)\circ A_{PL}(p_\lambda)\circ m_X \\
    &= A_{PL}(\Delta_r)\circ\xi\circ(id\otimes i_\lambda).
  \end{align*}
  Thus we obtain $m_X\circ(id\otimes \mu_r) = A_{PL}(\Delta_r)\circ\xi$ as stated.
\end{proof}

Next, we recall the concept of bigraded model introduced by Halperin and Stasheff \cite[Section 2]{HS}. For a formal space $F$, its minimal model $(\Lambda V, d)$ admits a structure of a bigraded algebra such that:
\begin{enumerate}
  \item the vector space $V$ has a second grading $V=\oplus_{i\ge 0}V_i$, which gives $\Lambda V$ the structure of a bigraded algebra,
  \item $d$ is homogeneous of lower degree $-1$, and
  \item $H_+(\Lambda V, d) = 0$ and $H(\Lambda V, d) = H_0(\Lambda V, d)$.
\end{enumerate}
\noindent
Such bigraded algebra $(\Lambda V, d)$ is called a \emph{bigraded model} of $F$. Remark that, for two bigraded models $(\Lambda V, d)$ and $(\Lambda W, d')$ of formal spaces $F$ and $F'$, the tensor product $(\Lambda V\otimes \Lambda W, d\otimes d')$ has a structure of a bigraded model of $F\times F'$, with the second grading $V\oplus W =\oplus_{i\ge 0}(V_i\oplus W_i)$.

A fibration $F \to X\to B$ is said to be \emph{totally noncohomologous to zero} with respect to a ring $R$, or simply \emph{TNCZ}, if the induced morphism $H^*(X; R) \to H^*(F; R)$ is surjective. Lupton \cite{L} utilized a bigraded model to prove the following proposition.

\begin{proposition}
  \cite[Proposition 3.2]{L}
  Let $F\to X\to B$ be a rational TNCZ fibration. If $F$ elliptic and formal and $B$ is formal, then $X$ is formal.
\end{proposition}
\noindent
In his proof, Lupton constructed a model
\[
(H(B), 0) \to (H(B)\otimes\Lambda V, D) \to (\Lambda V, d)
\]
such that:
\begin{enumerate}
  \item $(\Lambda V, d)$ is a bigraded model of $F$ with the second grading $V=\oplus_{i\ge 0}V_i$,
  \item $D(V_0)=0$ and
  \[
  D(V_i)\subset H(B)\otimes (\Lambda V)_{(i-1)} := \bigoplus_{j=0}^{i}(\Lambda V)_j.
  \]
\end{enumerate}
\noindent
We may call it \emph{a relative filtered model}, regarding it as a generalization of the filtered model introduced in \cite[Section 4]{HS}. We will see that the argument as in the proofs of \cite[Propositions 3.1 and 3.2]{L} can be applied to  the proof of the formality of $\Delta_r$.

\begin{lemma}
  \label{filtered}
  Let $F\to X\to B$ be a rational TNCZ fibration. If $F$ elliptic and formal and $B$ is formal, there is a commutative diagram
  \[
  \xymatrix{
  H(B) \ar[r]\ar@{=}[d]& (H(B)\otimes (\Lambda V)^{\otimes r},D) \ar[r]\ar[d]^{id\otimes \mu_r}& (\Lambda V)^{\otimes r}\ar[d]^{\mu_r}\\
  H(B) \ar[r]& (H(B)\otimes \Lambda V,D') \ar[r]& \Lambda V
  }
  \]
  where the two rows are the relative filtered model of
  \[
  F^r\to X^r_B\to B\quad\text{and}\quad F\to X\to B.
  \]
\end{lemma}

\begin{proof}
  Since $F$, $F^r$ and $B$ are formal, one can obtain a commutative diagram
  \begin{equation}
    \label{models}
    \xymatrix{
    A_{PL}(X^r_B) \ar[d]_{A_{PL}(\Delta_r)}& \mathcal{M}\otimes (\Lambda V)^{\otimes r} \ar[l]_{\simeq}^{\xi}\ar[r]^{\simeq}\ar[d]^{id\otimes\mu_r}& (H(B) \otimes (\Lambda V)^{\otimes r}, \delta) \ar[d]\\
    A_{PL}(X) & \mathcal{M}\otimes \Lambda V \ar[l]_{\simeq}^{m_X}\ar[r]^{\simeq}& (H(B) \otimes \Lambda V, \delta')
    }
  \end{equation}
  by Lemma \ref{diag comm}, which induces a commutative diagram
  \[
  \xymatrix{
  H(B) \ar[r]\ar@{=}[d]& (H(B)\otimes (\Lambda V)^{\otimes r},\delta) \ar[r]\ar[d]^{id\otimes \mu_r}& (\Lambda V)^{\otimes r}\ar[d]^{\mu_r}\\
  H(B) \ar[r]& (H(B)\otimes \Lambda V,\delta') \ar[r]& \Lambda V
  }
  \]
  where two rows are models for the two fibrations in the statement. Remark that $(\Lambda V)^{\otimes r}$ and $\Lambda V$ have structures of bigraded models of $F^r$ and $F$, respectively. Then by the argument in the proof of \cite[Proposition 3.1]{L}, the relative filtered model of the fibrations can be constructed by replacing the basis elements of $H(B)\otimes (\Lambda V)^{\otimes r}$ and $H(B)\otimes \Lambda V$.

  We now show that these two constructions are compatible. Recall that, if we denote the basis of $V$ by $\{v_\alpha\}_{\alpha\in I}$, then the basis elements of $V^{\oplus r}$ can be described by $\{i_{\lambda}(v_\alpha)\}_{\alpha\in I, 1\le \lambda \le r}$. Furthermore, the fibration $F\to X\to B$ is TNCZ if and only if the rational cohomology Serre spectral sequence collapses at the $E_2$-term, which ensures that the fibration $F^r \to X^r_B \to B$ is also TNCZ. Then since the composite $(id\otimes \Delta_r)\circ (id\otimes i_\lambda)$ is the identity, all basis elements $\{i_{\lambda}(v_\alpha)\}_{\alpha\in I, 1\le \lambda \le r}$ can be replaced without any loss of of naturality. Thus we obtain the diagram as stated.
\end{proof}

We are now ready to prove Theorem \ref{PTC nil ker}.

\begin{proof}
  [Proof of Theorem \ref{PTC nil ker}]
  By Lemma \ref{filtered}, we have a commutative diagram
  \[
  \xymatrix{
  A_{PL}(X^r_B) \ar[d]_{A_{PL}(\Delta_r)}& \mathcal{M}\otimes (\Lambda V)^{\otimes r} \ar[l]_{\simeq}^{\xi}\ar[r]^{\simeq}\ar[d]^{id\otimes\mu_r}& (H(B) \otimes (\Lambda V)^{\otimes r}, D) \ar[d]\\
  A_{PL}(X) & \mathcal{M}\otimes \Lambda V \ar[l]_{\simeq}^{m_X}\ar[r]^{\simeq}& (H(B) \otimes \Lambda V, D').
  }
  \]
  Then by the argument as in the proof of \cite[Proposition 3.2]{L}, we get
  \[
  H_{+}(H(B)\otimes (\Lambda V)^{\otimes r},D) = H_{+}(H(B)\otimes \Lambda V, D') = 0.
  \]
  We also obtain a commutative diagram
  \[
  \xymatrix{
  (H(B)\otimes (\Lambda V)^{\otimes r},\delta) \ar[r]^{\pi}\ar[d]& \frac{H(B)\otimes (\Lambda V)^{\otimes r}}{\delta (V_1^{\oplus r})} \ar@{=}[r]\ar[d] & H_{0}(H(B)\otimes (\Lambda V)^{\otimes r},\delta) \ar[d]^{\Delta_r^*}\\
  (H(B)\otimes \Lambda V,\delta') \ar[r]^{\pi'}& \frac{H(B)\otimes (H(B)\otimes \Lambda V)}{\delta' (V_1)} \ar@{=}[r] & H_{0}(H(B)\otimes \Lambda V,\delta'),
  }
  \]
  where the projections $\pi$ and $\pi'$ are quasi-isomorphisms. Combining this diagram with \eqref{models}, we have proved that $\Delta_r\colon X\to X^r_B$ is formal. Thus by \eqref{PTC Delta}, Propositions \ref{PTC quot} and \ref{secat formal}, the proof is finished.
\end{proof}

We exhibit some examples to show that the hypotheses in Theorem \ref{PTC nil ker} are essential.

\begin{example}
  \label{ex nonformal}
  Let $X$ be a rational space whose minimal model is given by
  \[
  \Lambda(x, y, z),\quad |x|=|y|=3,\,|z|=5,\quad dx=dy=0, dz=xy,
  \]
  and let $X\to X\to \ast$ be the trivial fibration. It is well known that the space $X$ is not formal. Kishimoto and Yamaguchi \cite{KY} computed $\TC_r(X)$ and obtained that, for $r\ge4$,
  \[
    \zcl_r[X\to \ast] \le 2r < 3(r-1) = \TC_r[X\to \ast] = \TC_r(X).
  \]
\end{example}

\begin{example}
  \label{ex hyperbolic}
  Let $F = (S^3\vee S^3)_\Q \to X\to B = (S^3\times S^5)_\Q$ be the TNCZ fibration as in \cite[III. 13]{T82} with a relative Sullivan model
  \begin{gather*}
    \Lambda(b, b') \to \Lambda((b, b', x, y, t, u, v)\oplus {V}^{\ge8}, d) \to \Lambda((x, y, t, u, v)\oplus {V}^{\ge8}),\\
    |b|=|x|=|y|=3, |b'|=|t|=5, |u|=|v|=7,\\ db=db'=dx=dy=0, dt=xy, du= tx-b'x, dv=ty.
  \end{gather*}
  \noindent
  Remark that the fiber and the base are formal, yet the fiber is hyperbolic. Thomas \cite{T82} proved that the total space $X$ is not formal.

  In this case, we have:
  \begin{equation}
    \label{CE2}
    \zcl_2[X\to B] = 2 < 3 \le \HTC_2[X\to B] \le \TC_2[X\to B].
  \end{equation}
  To see this, we first show that $\zcl_2[X\to B] = 2$. The fiberwise diagonal $\Delta_2\colon X\to X^2_{B}$ is modelled by
    \begin{gather*}
      \Delta_2^*\colon\Lambda((b, b', x_i, y_i, t_i, u_i, v_i)_{i=1, 2}\oplus (V^{\ge8})^{\oplus 2}, d')\to\Lambda((b, b', x, y, t, u, v)\oplus V^{\ge8}, d),\\
      \Delta_2^*(x_i)=x, \Delta_2^*(y_i)=y,\cdots,\Delta_2^*(v_i)=v
    \end{gather*}
    \noindent
    where the differential is given by
    \[
    d'b=d'b'=d'x_i=dy_i=0, d't_i=x_iy_i, d'u_i= t_ix_i-b'x_i \quad\text{and}\quad d'v_i=t_iy_i.
    \]
    Then for $1\le m\le8$, we have
    \[
    \mathrm{Ker}\,H^m(\Delta_2) =
    \begin{cases}
      \langle x_1-x_2, y_1-y_2 \rangle & (m = 3)\\
      \langle b(x_1-x_2), b(y_1-y_2), x_1y_2, x_2y_1 \rangle & (m = 6)\\
      \langle t_1x_1-b'x_2, t_2x_2-b'x_1, t_1y_2-t_2y_1, b'(y_1-y_2) \rangle & (m = 8)\\
      0 & (otherwise.)
    \end{cases}
    \]
    Thus the product of any three elements in $\mathrm{Ker} H^m(\Delta_2)$ is equal to zero. By degree reasons of the rational cohomology Serre spectral sequence for the fibration $F^2 \to X^2_B \to B$, one can deduce that $H^{\ge 15}(X^2_B) = 0$. Thus we obtain $\zcl_2[X\to B] = 2$. On the other hand, it is straightforward to see that the element $(x_1-x_2)(y_1-y_2)(t_1-t_2) \in \mathrm{Ker}\, (id\otimes\mu_2)^3$ is a nontrivial cocycle. Then we get $3 \le \HTC_2[X\to B]$ by Definition \ref{MTC and HTC}. Thus, by \eqref{PTC ineqs}, we obtain \eqref{CE2}.
\end{example}

\begin{example}
  \label{ex not TNCZ}
  Let $X$ be a rational space as in Example \ref{ex nonformal}, and consider a rational fibration $X\to B=S^3_\Q$ with a relative Sullivan model $\Lambda(x) \to \Lambda(x, y, z)$. Remark that the fiber has the rational homotopy type of $(S^3\times S^5)_\Q$, so is formal, and the fibration is not TNCZ.
  
  For $r\ge 5$, we aim to prove:
  \begin{equation}
    \label{CE3}
    \zcl_r[X\to B] \le 2r-3 < 2r-2 = \TC_r[X\to B].
  \end{equation}
    We first show the first inequality. The fiberwise diagonal $\Delta_r\colon X\to X^r_B$ is modelled by
    \begin{gather*}
      \Delta_r^*\colon\Lambda(x,y_1,\dots,y_r,z_1,\dots,z_r)\to \Lambda(x, y, z),\\
      \Delta_r^*(x)=x, \Delta_r^*(y_i)=y\quad\text{and}\quad\Delta_r^*(z_i)=z,
    \end{gather*}
    where the differential is given by $d(x)=d(y_i)=0$ and $d(z_i)=xy_i$. It is straightforward to see that any element that has word length $1$ and represents some element of $\mathrm{Ker}\, H^*(\Delta_r)$ is described as a linear sum
    \[
    \sum_{1\le i\le r-1} a_{i}(y_i-y_{i+1})\quad (a_{i}\in\Q).
    \]
    We also have that $(y_1-y_2)(y_2-y_3)\cdots(y_r-y_1)=0$. Since the word length of elements in  $\Lambda(x,y_1,\dots,y_r,z_1,\dots,z_r)$ are at most $2r+1$, we get
    \[
      \zcl_r[X\to B] \le (r-1)+\left\lfloor\frac{(2r+1)-(r-1)}{2}\right\rfloor = r +\left\lfloor \frac{r}{2}\right\rfloor \le 2r-3
    \]
    for $r\ge5$, which is the first inequality. On the other hand, Theorem \ref{main2} asserts that $\TC_r[X\to B] = 2r-2$, thus we obtain \eqref{CE3}.
\end{example}

We also remark that the hypothesis of Theorem \ref{PTC nil ker} is closely related to the conjecture of Halperin:
\begin{conjecture}(Halperin)
  \label{H conj}
    If $F$ is an elliptic space with the minimal model $\Lambda V$ such that $\dim V^{\mathrm{odd}} = \dim V^{\mathrm{even}}$, then any fibration $F\to X\to B$ is TNCZ.
\end{conjecture}
\noindent
Such spaces are called as $F_0$-\emph{spaces}, and known to be formal; for instance, even-dimensional sphere $S^{2n}$, homogeneous spaces and K\"{a}hler manifolds are $F_0$-spaces. Although there has been amount of studies on the affirmative cases of Halperin's conjecture, it is still open. For recent progress, see \cite{KW}.

%%%%% Section 5 %%%%%

\section{An upper bound by odd-degree extension}\label{Bound}

In this section, we give a proof of Theorem \ref{main 1}. We then establish an upper bound of $\TC_r[X\to B]$ as a consequence of Theorems \ref{PTC nil ker} and \ref{main 1}.

\begin{proof}[Proof of Theorem \ref{main 1}]
  We begin with the case which the kernel of the map $\pi_\ast(F)\to\pi_\ast(\widehat{F})$ is one-dimensional, say, of degree $2n+1$. Rationally, this case corresponds to a homotopy commutative diagram
  \[
  \xymatrix{
  K(\Q, 2n+1) = S^{2n+1}_\Q \ar@{=}[r]\ar[d]& S^{2n+1}_\Q \ar[rr]\ar[d]&& \ast \ar[d]\\
  F \ar[r]\ar[d]& X \ar[rr]\ar[d]&& B \ar@{=}[d]\\
  \widehat{F} \ar[r]& \widehat{X} \ar[rr]&& B,
  }
  \]
  where all columns and rows are rational fibrations. We set relative Sullivan models of the above diagram by
  \[
  \xymatrix{
  \Lambda\langle u \rangle && \Lambda\langle u \rangle \ar@{=}[ll]&& \Q \ar[ll]\\
  \Lambda V \ar[u]&& C\otimes\Lambda V \ar[ll]\ar[u]&& C \ar[ll]\ar[u]\\
  \Lambda \widehat{V} \ar[u]&& C\otimes\Lambda \widehat{V} \ar[ll]\ar[u]&& C,\ar[ll]\ar@{=}[u]
  }
  \]
  where $|u| = 2n+1$ and $V=\widehat{V}\oplus\langle u \rangle$.

  Suppose that $\TC_r[C\hookrightarrow C\otimes \Lambda\widehat{V}] = k$ and let $\hat{\iota}\colon C\otimes (\Lambda\widehat{V})^{\otimes r} \hookrightarrow C\otimes (\Lambda\widehat{V})^{\otimes r} \otimes \Lambda \widehat{U}$ denote a relative Sullivan model of the projection
  \[
  \rho_k\colon C\otimes (\Lambda\widehat{V})^{\otimes r} \twoheadrightarrow \frac{C\otimes (\Lambda\widehat{V})^{\otimes r}}{(\mathrm{Ker}\, (id\otimes \hat{\mu}))^{k+1}},
  \]
  where $\hat{\mu}\colon (\Lambda\widehat{V})^{\otimes r}\to \Lambda\widehat{V}$ denotes the multiplication.
  We also take a relative Sullivan model $\iota\colon C\otimes (\Lambda V)^{\otimes r} \hookrightarrow C\otimes (\Lambda V)^{\otimes r} \otimes \Lambda U$ of the projection
  \[
  \rho_{k+(r-1)}\colon C\otimes (\Lambda V)^{\otimes r} \twoheadrightarrow \frac{C\otimes (\Lambda V)^{\otimes r}}{(\mathrm{Ker}\, (id\otimes \mu))^{k+1+(r-1)}},
  \]
  where $\mu\colon (\Lambda V)^{\otimes r}\to \Lambda V$ denotes the multiplication.
  Then there is a commutative diagram
  \[
  \xymatrix{
  C\otimes (\Lambda V)^{\otimes r}= C\otimes (\Lambda \widehat{V})^{\otimes r}\otimes (\Lambda \langle u \rangle)^{\otimes r} \ar[rr]^{\hat{\iota}\otimes id} \ar[rrd]_{\rho_k \otimes id} \ar[d]_{\iota}&& C\otimes (\Lambda \widehat{V})^{\otimes r} \otimes \Lambda \widehat{U} \otimes (\Lambda \langle u \rangle)^{\otimes r}\ar[d]_{\simeq}\\
  C\otimes (\Lambda V)^{\otimes r} \otimes \Lambda U \ar[d]_{\simeq}&& \frac{C\otimes (\Lambda \widehat{V})^{\otimes r}}{(\mathrm{Ker}\, (id\otimes \hat{\mu}))^{k+1}} \otimes (\Lambda \langle u \rangle)^{\otimes r} \\
  \frac{C\otimes (\Lambda V)^{\otimes r}}{(\mathrm{Ker}\, (id\otimes \mu))^{k+1+(r-1)}} \ar@{.>}[rru].
  }
  \]
  Let $K$ and $\widehat{K}$ denote the kernel of $\mu$ and $\hat{\mu}$, respectively. We have
  \[
  \mathrm{Ker}\,(id\otimes\mu) = C \otimes (\widehat{K}\otimes(\Lambda \langle u \rangle)^{\otimes r} + \Lambda \widehat{V} \otimes K).
  \]
  Since $K^{r} = 0$ (cf. \cite[Section 4]{KY}), we get
  \[
  (\rho_k\otimes id)((\mathrm{Ker}\, (id\otimes \mu))^{k+1+(r-1)}) = 0,
  \]
  Then, in the above diagram, the homomorphism $\rho_k\otimes id$ can be factorized, and so, by an argument in the proof of \cite[Theorem 3.2]{HRV}, there is a homotopy retraction of $\rho_{k+(r-1)}$. Thus we proved that
  \begin{align*}
    \TC_r[X\to B]&=\TC_r[C\hookrightarrow C\otimes \Lambda V]\\
    &\le \TC_r[C\hookrightarrow C\otimes \Lambda \widehat{V}] + (r-1) = \TC_r[\widehat{X}\to B] + (r-1).
  \end{align*}

  For general cases, one can inductively apply the above argument because of the nilpotence condition of the relative Sullivan algebra, completing the proof.
\end{proof}

We next provide a sufficient condition for a rational fibration $F\to X\to B$ to admit an odd-degree extension, i.e., Theorem \ref{main 1} is applicable.
\begin{definition}
  \cite{HRV2}
  For a minimal model $\Lambda V$ of an elliptic space, we call an extension $\Lambda \widehat{V}\hookrightarrow \Lambda V$ as an \emph{$F_0$-basis extension} if ${\widehat{V}}^{\mathrm{even}} = V^{\mathrm{even}}$ and $\Lambda \widehat{V}$ is the minimal model of an $F_0$-space.
\end{definition}
\noindent
Let $F\to X\to B$ be a rational fibration such that:

\begin{enumerate}
  \item the minimal model $\Lambda V$ of $F$ admits an $F_0$-basis extension $\Lambda \widehat{V}\hookrightarrow \Lambda V$, and
  \item the fibration $F\to X\to B$ is \emph{pure}; it has a relative Sullivan model
  \[
  (C, d)\hookrightarrow(C\otimes \Lambda V, d)\to(\Lambda V, \overline{d})
  \]
  with $dV^{\mathrm{even}} = 0$ and $dV^{\mathrm{odd}} \subset C\otimes \Lambda V^{\mathrm{even}}$ (cf. \cite{T81}).
\end{enumerate}

\begin{lemma}
  \label{F0 ext}
  If a rational fibration $F\to X\to B$ satisfies the above assumptions, then it admits an odd-degree extension.
\end{lemma}

\begin{proof}
    We first show that $C\hookrightarrow C\otimes\Lambda\widehat{V}$ admits a structure of a relative Sullivan model. For $c\otimes 1\in C\otimes\Lambda\widehat{V}$, we clearly have $dc\otimes 1\in C\otimes\Lambda\widehat{V}$. Then we consider a homogeneous element $1\otimes v\in C\otimes\Lambda\widehat{V}$. If $|v|$ is even, then $d(1\otimes v)=0$ by the pureness of the model. If $|v|$ is odd, then we have
    \[
    d(1\otimes v) \in C\otimes \Lambda V^{\mathrm{even}} = C\otimes \Lambda \widehat{V}^{\mathrm{even}}
    \]
    Thus for any element $x \in C\otimes \Lambda\widehat{V}$, we get $dx\in C\otimes \Lambda\widehat{V}$.
    
    Let $W$ denote a subspace of $V$ such that $V = \widehat{V}\oplus W$. We now have a commutative diagram
    \[
    \xymatrix{
        C \ar@{^{(}->}[r]\ar@{=}[d]& C\otimes\Lambda\widehat{V} \ar@{->>}[r]\ar@{^{(}->}[d]& \Lambda\widehat{V} \ar@{^{(}->}[d]\\
        C \ar@{^{(}->}[r]\ar@{->>}[d]& C\otimes\Lambda V \ar@{->>}[r]\ar@{->>}[d]& \Lambda V \ar@{->>}[d]\\
        \Q \ar[r]& \Lambda W \ar@{=}[r]& \Lambda W,
    }
    \]
    where all columns and rows in the top left square are relative Sullivan algebras. Recall that $W$ is of finite dimension and concentrated in odd degrees. Thus, by the spatial realization of the above diagram, the statement is proved.
\end{proof}

We now combine Theorems \ref{PTC nil ker} and \ref{main 1}. Let $F\to X\to B$ be a rational fibration such that:

\begin{enumerate}
  \item $B$ is formal and $F$ is elliptic formal,
  \item the minimal model $\Lambda V$ of $F$ admits an $F_0$-basis extension $\Lambda \widehat{V}\hookrightarrow \Lambda V$, and
  \item there is a \emph{pure} relative Sullivan model $C\hookrightarrow C\otimes\Lambda V\to \Lambda V$.
\end{enumerate}

\begin{corollary}
  \label{main 2}
  If a rational fibration $F\to X\to B$ satisfies the above assumptions, then there is an inequality
  \[
  \TC_r[X\to B] \le \zcl_r[\widehat{X}\to B] + m(r-1),
  \]
  where $m$ denotes the codimension of $\widehat{V}$ in $V$.
\end{corollary}

\begin{proof}
  By Lemma \ref{F0 ext}, one can assume that $F\to X\to B$ is an odd-degree extension of a rational fibration, say, $\widehat{F}\to \widehat{X}\to B$. Then by Theorem \ref{main 1}, we have
  \[
  \TC_r[X\to B] \le \TC_r[\widehat{X}\to B] + m(r-1).
  \]
  Remark that $\widehat{F}$ is formal and elliptic, $B$ is formal and $\widehat{F} \to \widehat{X} \to B$ is also a pure fibration. Then by \cite[Theorem 2]{T81}, the fibration $\widehat{F} \to \widehat{X} \to B$ is TNCZ. Thus we can apply Theorem \ref{PTC nil ker}, completing the proof.
\end{proof}

\begin{example}
  Let $V_q(\R^{m+q})$ denote the space of the orthonormal $q$-frames in $\R^{m+q}$. Then the projection of the first vector admits a structure of a $2$-frame bundle
  \[
  V_2(\R^{2n+2}) \to V_3(\R^{2n+3}) \xrightarrow{p} S^{2n+2}.
  \]
  We denote its rationalization by $F\to X\xrightarrow{p} B$.

  We show that, for $n\ge2$, there is an identity
  \begin{equation}
    \label{E4}
    \TC_2[X\to B] = 3.
  \end{equation}
  Recall that $F$ has the same homotopy type as $(S^{2n}\times S^{2n+1})_\Q$, and so we get
  \[
  3 = \TC((S^{2n}\times S^{2n+1})_\Q) = \TC(F) \le \TC_2[X\to B].
  \]
  We now construct a relative Sullivan model of $F\to X\to B$. Recall that there is a commutative diagram
  \[
  \xymatrix{
  S^{2n} \ar[r]\ar@{=}[d]& V_2(\R^{2n+2}) \ar[r]\ar[d]& S^{2n+1}\ar[d]\\
  S^{2n} \ar[r]\ar[d]& V_3(\R^{2n+3}) \ar[r]\ar[d]^{p}& V_2(\R^{2n+3})\ar[d]^{p'}\\
  \ast \ar[r]& S^{2n+2} \ar@{=}[r]& S^{2n+2}
  }
  \]
  such that all columns and rows are fiber bundles, where $p'$ is the projection of first vector. Since the right column is a sphere bundle of the Euler characteristic $2$, it is modelled as
  \begin{gather*}
    \Lambda(a,b) \hookrightarrow \Lambda(a,b,y) \to \Lambda(y),\\
    |a|=2n+2, |b|=4n+3, |y|=2n+1,\\
    da=dx=0, db=a^2, dy=2a.
  \end{gather*}
  Then by degree reasons, we get a relative Sullivan model of $F\to X\to B$ by
  \begin{gather*}
    \Lambda(a,b) \hookrightarrow \Lambda(a,b,x,y,z) \to \Lambda(x,y,z),\\
    |a|=2n+2, |b|=4n+3, |x|=2n, |y|=2n+1, |z|=4n-1,\\
    da=dx=0, db=a^2, dy=2a, dz=x^2.
  \end{gather*}
  One can easily confirm that the rational fibration $F\to X\to B$ satisfies all assumptions in Corollary \ref{main 2}; then it is an odd-degree extension of a rational fibration $\widehat{F}\to \widehat{X}\to B$ which has a relative Sullivan model
  \[
  \Lambda(a,b) \hookrightarrow \Lambda(a,b,x,z) \to \Lambda(x,z).
  \]
  Remark that this model is trivial; then we get $\TC_2[\widehat{X}\to B]= \zcl_2[\widehat{X}\to B] = 2$. Thus we obtain
  \[
  \TC_2[X\to B] \le \zcl_2[\widehat{X}\to B] + 1(2-1) = 3,
  \]
  and so \eqref{E4}.
\end{example}

\begin{remark}
  In the above example, Corollary \ref{main 2} provides a sharper upper bound than those previously known, such as the rational category of $X^2_B$, which is $4$, or that given in \cite[Lemma 1]{FKS}.
\end{remark}

%%%%% Section 5 %%%%%

\section{The fiber concentrated in odd degrees}\label{Odd fiber}

In this section, we focus on a rational fibration $F\to X\to B$ of finite rational type such that $F$ is elliptic and concentrated in odd degrees, and prove Theorem \ref{main 2}. We then see that Theorem \ref{main 2} is in connection with other results on the integral sequential parametrized topological complexity.

We first recall the following lemma, which is a generalization of \cite[Theorem 1.4]{JMP}.

\begin{lemma}
  \label{TC odd fiber}
  \cite[Theorem 1.4]{KY}
  For $F$ above, there are identities
  \[
  \TC_r(F) = (r-1)\cat(F) = (r-1)\dim(\pi_\mathrm{odd}(F)).
  \]
\end{lemma}

\begin{proof}[Proof of Theorem \ref{main2}]
  Let $\mathcal{M}\hookrightarrow \mathcal{M}\otimes \Lambda V\to \Lambda V$ be a relative minimal Sullivan model of $F\to X\to B$. Then since $\Lambda V$ is the minimal model of $F$, the graded vector space $V$ is of finite dimension and concentrated in odd degrees. Furthermore, there is a commutative digram
  \[
  \xymatrix{
    \mathcal{M} \ar[r] \ar@{=}[d] & \mathcal{M}\otimes\Q \ar[r] \ar@{^{(}-{>}}[d] & \Q\ar@{^{(}->}[d] \\
    \mathcal{M} \ar[r]  & \mathcal{M}\otimes\Lambda V \ar[r] & \Lambda V
  }
  \]
  where the columns are relative Sullivan algebras. Then by Theorem \ref{main 1}, we get
  \[
  \TC_r[X\to B] \le \TC_r[id\colon X\to X]+(r-1)\dim(V) = (r-1)\dim(\pi_\mathrm{odd}(F)).
  \]
  On the other hand, we have
  \[
  \TC_r[X\to B] \ge  \TC_r(F) = (r-1)\dim(\pi_\mathrm{odd}(F))
  \]
  by Lemmas \ref{bounds 1} and \ref{TC odd fiber}. Thus the proof is complete.
\end{proof}

\begin{remark}
  The upper bound in the above proof can also be obtained from the nilpotency of the ideal $\mathrm{Ker}\,(id\otimes \mu_r)$ in \eqref{diag model}. For detail, see \cite[Proposition 21]{CV15}.
\end{remark}

Theorem \ref{main2} can be interpreted in several contexts. First, it is a rational analogue of the result of Farber and Paul on the sequential parametrized topological complexity of a principal $G$-bundle $X\to B$, where $G$ is a connected topological group.

\begin{proposition}
  \cite[Proposition 3.3]{FP}
  There are identities
  \[
  \TC_r[X\to B] = \cat(G^{r-1}) = \TC_r(G).
  \]
\end{proposition}
\noindent
Since an elliptic $H$-space has the minimal model $(\Lambda V, 0)$ with $V$ of finite dimension and concentrated in odd degrees, the rationalized version of the above proposition is the special case of Theorem \ref{main2}.

The theorem can also be related to other upper bounds of the parametrized topological complexity of spherical fibrations. Let $X\to B$ be the unit sphere bundle of a real vector bundle of rank $n+1$, and $\widetilde{X}\to B$ be the bundle of orthonormal $2$-frames. Then there is a sphere bundle $S^{n-1}\to\widetilde{X}\to X$.

\begin{proposition}
  \cite[Theorem 14]{FW1}
  \label{upper bound FW}
  There is an inequality
  \[
    \TC_2[X\to B]\le\mathsf{secat}(\widetilde{X}\to X)+1.
  \]
\end{proposition}
\noindent
In \cite{M}, the author generalized the above upper bound and obtained:

\begin{proposition}
  \cite[Proposition 2.7]{M}
  Let $S^{2n+1}\to X\to\Sigma B$ be a homotopy fibration with $B$ connected. Then There is an inequality
  \[
  \TC_2[X\to \Sigma B]\le 2.
  \]
\end{proposition}

\begin{remark}
  Let $S^{4n-1}\to T\to S^{4n}$ be the unit tangent bundle, and let $F\to X\to B$ be its rationalization. In \cite{M}, the author proved that $\TC_2[T\to S^{4n}] = 2$, while we have $\TC_2[X\to B] = 1$ by Theorem \ref{main2}. One may consider that the difference derives from the $2$-component of the Hopf invariant, which vanishes after the rationalization.
\end{remark}

%%%%% Section 6 %%%%%

\section{A generalization of the $\TC$-generating function}\label{PTC function}

In this section, we provide two classes of fibrations which support Question \ref{PTC growth}. The first case is a direct consequence of Theorem \ref{main2}, and so the proof is omited.

\begin{corollary}
  Let $F\to X\xrightarrow{p}B$ be a rational fibration of finite rational type such that $F$ is elliptic and concentrated in odd degrees. Then $\mathcal{F}_{p}(z)$ a rational function $\displaystyle \frac{\cat(F)}{(1-z)^2}$.
\end{corollary}

The second case is a consequence of Theorem \ref{PTC nil ker}, and is related to the Halperin conjecture.

\begin{proposition}
  \label{TC gen even}
  Let $F= S^{2n}_\Q$ or $\C P^n_\Q$ for some $n$, and let $F\to X\xrightarrow{p}B$ a rational fibration such that $B$ is formal. Then $\mathcal{F}_{p}(z)$ is a rational function
  \[
  \frac{P(z)}{(1-z)^2}
  \]
  such that $P(z)$ is a polynomial with
  \[
  P(1)= \cat(F) =
  \begin{cases}
    1 & (F = S^{2n}_\Q)\\
    n & (F = \C P^n_\Q).
  \end{cases}
  \]
\end{proposition}

We prepare a lemma to evaluate the growth of $\TC_{r}[X\to B]$, which is a generalization of \cite[Lemma 2]{FKS}. Let $\cupl(Y)$ denote the cuplength of a rational space $Y$, i.e., $\nil\,H^*(Y)$. It is well known that, for a formal $Y$, there is an identity $\cupl(Y) = \cat(Y)$.

\begin{lemma}
  \label{diff nil}
  Let $F\to X\to B$ be a rational fibration. Then there is an inequality
  \[
  \cupl(F) \le \zcl_{r+1}[X\to B]-\zcl_r[X\to B].
  \]
\end{lemma}

\begin{proof}
  Let $\mathcal{M} \hookrightarrow \mathcal{M}\otimes \Lambda V\to \Lambda V$ denote the relative minimal Sullivan model, and let $\zcl_r[X\to B]=k$ and $\cupl(F)=l$. We take representatives $a_1, \cdots, a_k \in \mathcal{M}\otimes (\Lambda V)^{\otimes r}$ and $b_1, \cdots, b_l \in \Lambda V$. Let $\varpi\colon X^{r+1}_B\to X^r_B$ denote the projection of the first $r$ coordinates, which is modelled by
  \[
    \mathcal{M}\otimes (\Lambda V)^{\otimes r} \to \mathcal{M}\otimes (\Lambda V)^{\otimes {r+1}},\quad a\mapsto \bar{a} := a\otimes1.
  \]
  Then $\bar{a}_i$ represents a nonzero element in $\mathrm{Ker}\, H^*(\Delta_{r+1})$. Now we set
  \[
  \bar{b}_i:= 1\otimes1\otimes\cdots\otimes b_i - 1\otimes b_i\otimes\cdots\otimes 1 \in \mathcal{M}\otimes (\Lambda V)^{\otimes {r+1}},
  \]
  which is a nontrivial cocycle and contained in $\mathrm{Ker}\,(id\otimes\mu)$.

  We will prove that $\bar{a}_1\cdots\bar{a}_k\bar{b}_1\cdots\bar{b}_l$ represents a nonzero element in $\mathrm{Ker}\, H^*(\Delta_{r+1})$, from which the statement follows. We consider a dg-module homomorphism $\psi\colon (\Lambda V)\to \Q$ which satisfies
  \[
  \psi(b_1\cdots b_l)=1,\quad \psi(b)=0 \quad\text{for}\quad |b|\neq|b_1\cdots b_l|.
  \]
  Then $id\otimes\psi\colon \mathcal{M}\otimes (\Lambda V)^{\otimes {r+1}}\to \mathcal{M}\otimes (\Lambda V)^{\otimes r}$ is dg-module homomorphism and
  \[
  (id\otimes\psi)(\bar{a}_1\cdots\bar{a}_k\bar{b}_1\cdots\bar{b}_l) = a_1\cdots a_k
  \]
  represents a nonzero element in $\mathrm{Ker}\, H^*(\Delta_{r})$ by the assumption. Thus the proof is finished.
\end{proof}

\begin{proof}[Proof of Proposition \ref{TC gen even}]
  By the result of Shiga and Tezuka \cite{ST}, we have that $S^{2n}$ and $\C P^n$ satisfy the Halperin conjecture. Then the rational fibration $F\to X\to B$ is TNCZ and, by Theorem \ref{PTC nil ker} and Lemma \ref{diff nil}, we have
  \[
  \cat(F)\le \zcl_{r+1}[X\to B]-\zcl_r[X\to B] = \TC_{r+1}[X\to B]-\TC_{r}[X\to B].
  \]
  Remark that $\cat(S^{2n}_\Q)=1$ and $\cat(\C P^n_\Q)=n$. On the other hand, by \cite[Theorem 5]{Sv}, the upper bound is evaluated by
  \begin{align*}
    \TC_r[E\to B]&\le \frac{r\dim F + \dim B + 1}{2} \\
    &=
    \begin{cases}
      r + \displaystyle\tfrac{\dim B + 1}{2} & (F = S^{2n}_\Q)\\
      rn + \displaystyle\tfrac{\dim B + 1}{2} & (F = \C P^n_\Q).
    \end{cases}
  \end{align*}
  \noindent
  Thus by \cite[Lemma 1]{FKS}, the statement is proved.
\end{proof}

\end{document}